\documentclass[a4paper,11pt,twoside]{amsart}

\usepackage[latin1]{inputenc}
\usepackage{amsmath, amsfonts, amssymb,amsthm}
\usepackage[mathscr]{eucal}
\usepackage[pdftex]{graphicx}
\usepackage{color}
\usepackage[T1]{fontenc}
\usepackage[sc]{mathpazo}
\usepackage[all]{xy}
\usepackage{hyperref}
\usepackage{enumerate}

\newcommand{\R}{\mathbb{R}}

\newcommand{\s}{\mathbb{S}}
\newcommand{\h}{\mathbb{H}}
\newcommand{\E}{\mathbb{E}}
\newcommand{\M}{\mathbb{M}}

\newcommand{\Sol}{\mathrm{Sol}_3}

\newcommand{\X}{\mathfrak{X}}

\newcommand{\df}{\,\mathrm{d}}

\newcommand{\ad}{\mathrm{ad}}

\newtheorem{theorem}{Theorem}[section]
\newtheorem{proposition}[theorem]{Proposition}
\newtheorem{corollary}[theorem]{Corollary}
\newtheorem{lemma}[theorem]{Lemma}

\theoremstyle{definition}

\theoremstyle{remark}
\newtheorem{remark}[theorem]{Remark}

\numberwithin{equation}{section}

\title[Totally umbilical surfaces in homogeneous $3$-manifolds]{The classification of totally umbilical surfaces in homogeneous $3$-manifolds}

\author{Jos\'{e} M. Manzano}
\address{Departamento de Geometr\'{\i}a  y Topolog\'{\i}a, Universidad de Granada. Avda. Fuen\-tenueva s/n, 18071 Granada, Spain}
\email{jmmanzano@ugr.es}

\author{Rabah Souam}
\address{Institut de Math\'{e}matiques de Jussieu-Paris Rive Gauche,   UMR 7586, B\^{a}timent Sophie Germain,  Case 7012, 75205  Paris Cedex 13, France}
\email{souam@math.jussieu.fr}

\thanks{The first author was partially supported by the Spanish research project MTM2011-22547 and the Junta de Andaluc\'{i}a Grant P09-FQM-5088. The second author  was partially supported by the ANR-11-IS01-0002 grant.}



\begin{document}
\maketitle

\begin{abstract}
We obtain an exhaustive classification of totally umbilical surfaces in unimodular and non-unimodular simply-connected $3$-dimen\-sional Lie groups endowed with arbitrary left-invariant Riemannian metrics. This completes the classification of totally umbilical surfaces in homogeneous Riemannian $3$-manifolds.
\end{abstract}

\section{Introduction}\label{sec:introduction}

A submanifold of a Riemannian manifold is said \emph{totally umbilical} if the second fundamental form is proportional to the induced metric. In the particular case the second fundamental form identically vanishes, then it is called \emph{totally geodesic}. Totally umbilical submanifolds play an important role in submanifold theory and have been extensively studied in many ambient manifolds (e.g., see~\cite{Chen,Chen2,MT,ST09,SV,Tsu96,Veken}). The main aim of this paper is to classify totally umbilical surfaces in homogeneous Riemannian $3$-manifolds (i.e., $3$-manifolds on which the isometry group acts transitively). Surface theory in homogeneous $3$-manifolds is currently a very active research topic and determining the totally umbilical surfaces in these spaces is a basic question in the theory.

In the study of homogeneous manifolds, Lie group theory appears in a very natural way, since any Lie group endowed with a left-invariant metric is homogenous. This is due to the fact that the left-invariance of the metric implies that left-translations are isometries. In the case of dimension $3$, the converse is almost true as the following result shows (cf.~\cite[Theorem~2.4]{MPR}):
\begin{quotation}
Any simply-connected homogeneous Riemannian $3$-mani\-fold is isometric to a $3$-dimensional Lie group endowed with a left-invariant metric, except for the product manifolds $\s^2(\kappa)\times\R$, where $\s^2(\kappa)$ stands for the $2$-sphere of constant Gaussian curvature $\kappa>0$.
\end{quotation}
Simply-connected $3$-dimensional Lie groups endowed with a left-invariant metric will be called \emph{metric Lie groups} in the sequel, and their isometry groups have dimension $3$, $4$ or $6$. We refer the reader to~\cite{MPR} for an exhaustive introduction to metric Lie groups (see also~\cite{Milnor} and Section~\ref{sec:metricliegroups}). 

We remark that non-simply-connected homogeneous $3$-manifolds are just Riemannian quotients of the simply-connected ones. As the projection over such a quotient defines a local isometry, totally umbilical surfaces in the quotient can be lifted to totally umbilical surfaces in the universal cover, so there will be no loss of generality in considering the ambient spaces to be simply-connected. More generally, conformal diffeomorphisms preserve totally umbilical surfaces, so rescaling the metric will not affect the discussion of totally umbilical surfaces, either.

We will now summarize some results on totally umbilical surfaces that are already known for some simply-connected homogeneous $3$-manifolds. 
\begin{itemize}
 \item If the isometry group has dimension $6$, then they have constant sectional curvature. In $\R^3$, totally umbilical surfaces are planes (totally geodesic) and round spheres. In $\s^3$, totally umbilical surfaces are round spheres (totally geodesic if and only if they are great spheres). Finally, in $\h^3$, totally umbilical surfaces are totally geodesic planes or their equidistant surfaces, round spheres, and horospheres (see also~\cite{Spivak}).
 \item Those whose isometry group has dimension $4$ are classified in a $2$-parameter family $\E(\kappa,\tau)$, where $\kappa,\tau\in\R$. In fact, non-negative constant sectional curvature spaces are also contained in this family for $\kappa-4\tau^2=0$ (see~\cite{Daniel} for a detailed description). $\E(\kappa,\tau)$-spaces are also characterized in~\cite{Man} as Killing submersions over $\M^2(\kappa)$ with constant bundle curvature $\tau$. 
\begin{itemize}
 \item If $\tau=0$, they reduce to the Riemannian product spaces $\s^2(\kappa)\times\R$ and $\h^2(\kappa)\times\R$. They are locally conformally $\R^3$, from where totally umbilical surfaces can be studied (see also~\cite{ST09}).
  \item If $\tau\neq 0$, then it is shown in~\cite{ST09} that these ambient spaces do not admit totally umbilical surfaces.
\end{itemize}
 \item Finally, if the dimension of the isometry group is $3$, little is known about totally umbilical surfaces, except for totally geodesic ones in the case the metric Lie group is unimodular, which have been classified in~\cite{Tsu96}, and totally umbilical ones in the $\Sol$ group endowed with its standard metric, which have been classified in~\cite{ST09}.
\end{itemize}

In this paper, we undertake the rest of metric Lie groups, which completes the classification in all homogenous $3$-manifolds. Metric Lie groups are divided into two families: unimodular (when its left-invariant Haar measure is also right-invariant) and non-unimodular ones.  The results in this paper are summarized as follows:
\begin{enumerate}
 \item Unimodular metric Lie groups form a $3$-parameter family which contains the family $\E(\kappa,\tau)$. In Section~\ref{sec:unimodular}, we will show that, except for $\R^3$, $\s^3$, $\Sol$, and the totally geodesic examples that appear in some special cases (given by~\cite{Tsu96}), there exist no totally umbilical surfaces (see Theorem~\ref{thm:unimodular}). In particular, we extend the corresponding results in~\cite{ST09,Tsu96} giving alternative proofs for them.
 \item Non-unimodular metric Lie groups also form a $3$-parameter family. In Section~\ref{sec:nounimodular}, we will prove a non-existence result of totally umbilical surfaces in non-unimodular metric Lie groups different from $\h^3$, $\h^2(\kappa)\times\R$, and some special cases which admit two families of totally geodesic surfaces and two families of totally umbilical which are not totally geodesic (see Theorem~\ref{thm:nounimodular}).
\end{enumerate}

Note that if the set of fixed points of a non trivial isometry of $G$ contains a surface, then this surface is totally geodesic. We remark that in  the exceptional families of totally geodesic surfaces appearing in the list above, those in the non-unimodular case are sets of fixed points of certain mirror symmetries, whereas in the unimodular case they are not (see also Example 2.23 and Proposition 2.24 in~\cite{MPR}). It is also interesting to point out that totally umbilical surfaces which are not totally geodesic only exist on those spaces which admit mirror symmetries (see also Remark~\ref{rmk:family-existence}). In fact, totally umbilical surfaces which are not totally geodesic turn out to be invariant by some mirror symmetry.

As a consequence, we get that the only $3$-dimensional homogeneous spaces which are locally conformally flat are those with constant sectional curvature and   the Riemannian product spaces $\h^2(\kappa)\times\R$ and $\s^2(\kappa)\times\R$ together with their Riemannian quotients.

Along the paper, totally umbilical surfaces are supposed to be smooth, though the involved arguments work when only $\mathcal C^3$-regularity is assumed.

\section{Preliminaries on metric Lie groups}\label{sec:metricliegroups}

As mentioned in the introduction, a \emph{metric Lie group} is a Lie group equipped with a left-invariant metric or, equivalently, a metric for which left-translations are isometries. In the sequel, we will suppose that all metric Lie groups are simply-connected.

A Lie group $G$ is called \emph{unimodular} if its left-invariant Haar measure is also right-invariant. It is well-known that $G$ is unimodular if, and only if, for any $X\in\mathfrak g$, the endomorphism $\ad_X:\mathfrak g\to\mathfrak g$ given by $\ad_X(Y)=[X,Y]$ has trace equal to zero (here, $\mathfrak g$ denotes the Lie algebra associated to $G$). On the other hand, both the cross product $\wedge$ and the Lie bracket $[\cdot,\cdot]$ are skew-symmetric operator defined on $\mathfrak g\times\mathfrak g$ so there exists a unique linear operator $L:\mathfrak g\to\mathfrak g$ such that
\begin{equation}\label{eqn:unimodular}
[X,Y]=L(X\wedge Y),\quad\text{for all }X,Y\in\mathfrak g.
\end{equation}
Thus $G$ is unimodular if and only if $L$ is self-adjoint~\cite[Lemma 4.1]{Milnor}. We will now discuss the unimodular and non-unimodular cases separately. 

\subsection{Unimodular metric Lie groups}
Let $G$ be a $3$-dimensional unimodular Lie group endowed with a left-invariant Riemannian metric $\langle\cdot,\cdot\rangle$. Since the operator $L$ given by~\eqref{eqn:unimodular} is self-adjoint, there exist a left-invariant orthonormal frame $\{E_1,E_2,E_3\}$ in $G$ and $c_1,c_2,c_3\in\R$ such that 
\begin{align}
[E_1,E_2]&=c_3E_3, &[E_2,E_3]&=c_1E_1,&[E_3,E_1]&=c_2E_2.\label{eqn:lie-bracket-unimodular-frame}
\end{align}
The constants $c_1,c_2,c_3\in\R$ determine  both the geometry and the underlying Lie group structure. Note that, if we change the sign of all the $c_i$, then the geometry of $G$ is preserved but its orientation is reversed so the aforementioned structure is invariant under a global change of signs of the constants. The list of underlying Lie groups is given by the figure~\ref{fig:unimodular-mlg}.

\begin{figure}
\begin{tabular}{|c|c|}
\hline Signs of $c_1$, $c_2$, $c_3$&Simply-connected Lie group\\\hline\hline
$+$, $+$, $+$&$\mathrm{SU}(2)$\\
$+$, $+$, $-$&$\widetilde{\mathrm{Sl}}_2(\R)$\\
$+$, $+$, $\ 0$&$\widetilde E(2)$\\
$+$, $-$, $\ 0$&$\mathrm{Sol}_3$\\
$+$, $\ 0$, $\ 0$&$\mathrm{Nil}_3$\\
$0$, $\,\,0$, $\,\,0$&$\R^3$\\\hline
\end{tabular}
\caption{Three-dimensional simply-connected unimodular metric Lie groups in terms of the signs of the structure constants.}\label{fig:unimodular-mlg}
\end{figure}

Although such Lie group classification only depends on the signs of the structure constants, their values determine all the left-invariant metrics they carry. In other words, the values of $c_1, c_2, c_3$ determine the metric structure of $G$. For instance, multiplying all the $c_i$ by a positive constant leads to a metric homothetical to the original one.

Let us now consider the real numbers $\mu_1,\mu_2,\mu_3$ given by 
\[\mu_1=\tfrac{1}{2}(-c_1+c_2+c_3),\quad\mu_2=\tfrac{1}{2}(c_1-c_2+c_3),\quad\mu_3=\tfrac{1}{2}(c_1+c_2-c_3).\]
By using the Koszul formula, it is easy to check that the Levi-Civita connection $\overline\nabla$ on $G$ satisfies
\begin{equation}\label{eqn:nabla-unimodular}
\begin{array}{lclcl}
\overline\nabla_{E_1}E_1=0,&&\overline\nabla_{E_1}E_2=\mu_1E_3,&&\overline\nabla_{E_1}E_3=-\mu_1E_2,\\
\overline\nabla_{E_2}E_1=-\mu_2E_3,&&\overline\nabla_{E_2}E_2=0,&&\overline\nabla_{E_2}E_3=\mu_2E_1,\\
\overline\nabla_{E_3}E_1=\mu_3E_2,&&\overline\nabla_{E_3}E_2=-\mu_3E_1,&&\overline\nabla_{E_3}E_3=0.
\end{array}
\end{equation}

\begin{remark}
The system of linear equations defining $\mu_i$ in terms of the $c_i$ is invertible so there is no loss of generality in studying unimodular metric Lie groups in terms of $\mu_i$. Nevertheless, the description of the underlying Riemannian $3$-manifolds is better understood by using $c_i$. We also remark that, given $i,j\in\{1,2,3\}$, the condition $c_i\leq c_j$ is equivalent to $\mu_j\leq\mu_i$, which will be useful in the sequel.
\end{remark}

From the expression of the Levi-Civita connection it is easy to compute the Riemannian curvature tensor $R$ of $G$, which satisfies
\[\begin{array}{l}
R(E_1,E_2)E_1=(\mu_1\mu_2-c_3\mu_3)E_2,\\
R(E_1,E_2)E_2=-(\mu_1\mu_2-c_3\mu_3)E_1,\\
R(E_1,E_2)E_3=0,\\
R(E_1,E_3)E_1=(\mu_1\mu_3-c_2\mu_2)E_3,\\
R(E_1,E_3)E_2=0,\\
R(E_1,E_3)E_3=-(\mu_1\mu_3-c_2\mu_2)E_1,\\
R(E_2,E_3)E_1=0,\\
R(E_2,E_3)E_2=(\mu_2\mu_3-c_1\mu_1)E_3,\\
R(E_2,E_3)E_3=-(\mu_2\mu_3-c_1\mu_1)E_2.
\end{array}\]

\begin{lemma}\label{lemma:R-unimodular}
In the previous situation,
\begin{align*}
R&=(\mu_2\mu_3-c_1\mu_1)R_1+(\mu_1\mu_3-c_2\mu_2)R_2+(\mu_1\mu_2-c_3\mu_3)R_3,
\end{align*}
where, for $i\in\{1,2,3\}$ and $X,Y,Z\in\X(G)$, the tensor $R_i$ is given by
\begin{align*}
R_i(X,Y)Z&=\langle X,Z\rangle Y-\langle Y,Z\rangle X-\langle Z,E_i\rangle\langle X,E_i\rangle Y\\&\quad+\langle Z,E_i\rangle\langle Y,E_i\rangle X-\langle Y,E_i\rangle\langle X,Z\rangle E_i+\langle X,E_i\rangle\langle Y,Z\rangle E_i.
\end{align*}
\end{lemma}

\begin{proof}
It suffices to check that this tensor coincides with $R$ on the basis $\{E_1,E_2,E_3\}$, which is a straightforward computation.
\end{proof}

The scalar curvature of $G$ is constant and will be denoted by $\rho$. It can be computed by using Lemma~\ref{lemma:R-unimodular} as
\begin{equation}\label{eqn:rho-unimodular}
\rho=\sum_{i,j=1}^3\langle R(E_i,E_j)E_j,E_i\rangle=2(\mu_1\mu_2+\mu_1\mu_3+\mu_2\mu_3).
\end{equation}

\begin{remark}\label{rmk:unimodular-examples}
Let us analyze the case where two of the structure constants are equal, so we will assume $c_1=c_2$ without loss of generality. If $c_3\neq 0$, then there exist $\kappa,\tau\in\R$ such that $c_1=c_2=\frac{\kappa}{2\tau}$ and $c_3=2\tau$. It can be shown that the metric Lie group $G$ with structure constants $c_1,c_2,c_3$ is isometric to the space $\E(\kappa,\tau)$, whose isometry group has dimension $4$ or $6$ (see also~\cite{Daniel}). If, on the contrary, $c_3=0$, then $G$ is isomorphic to $\R^3$ for $c_1=c_2=0$ or isomorphic to $\widetilde E(2)$, the group of orientation-preserving rigid motions of the Euclidean plane $\R^2$, endowed with its standard metric, for $c_1=c_2\neq 0$. It is interesting to observe that $\widetilde E(2)$ and $\R^3$ are isometric but the underlying Lie group structures are not isomorphic.

We also remark that if the constants $c_1,c_2,c_3$ are different, then the isometry group of $G$ has dimension $3$. A special case is the $\Sol$ group with its standard metric, which is obtained for $c_1=1$, $c_2=0$, $c_3=-1$.
\end{remark}

\subsection{Non-unimodular metric Lie groups}\label{sec:nouni-intro}
A natural way to provide $3$-dimensional Lie groups is to consider semidirect products $\R^2\ltimes_A\R$, where $A$ is a $2\times 2$ real matrix. Such a group structure is given by 
\[(p_1,z_1)\star(p_2,z_2)=(p_1+e^{z_1A}p_2,z_1+z_2),\quad (p_1,z_1),(p_2,z_2)\in\R^2\times\R,\]
where $e^{zA}=\sum_{k=0}^\infty\frac{z^kA^k}{k!}$ denotes the exponential matrix.

Up to rescaling the metric, every non-unimodular metric Lie group is isometric to $\R^2\ltimes_{A(a,b)}\R$, where 
\begin{equation}\label{eqn:A-nonuni}
A(a,b)=\left(\begin{matrix}(1+a)&-(1-a)b\\(1+a)b&1-a\end{matrix}\right),
\end{equation}
for some constants $a,b\geq 0$, endowed with the left-invariant metric determined by the fact that
\begin{align*}
 E_1&=\alpha_{11}(z)\partial_x+\alpha_{21}(z)\partial_y,&E_2&=\alpha_{12}(z)\partial_x+\alpha_{22}(z)\partial_y,&E_3&=\partial_z,
\end{align*}
defines an orthonormal frame. Here, $\alpha_{ij}(z)$ denote the entries of $e^{zA}$ (see~\cite[Section 2.5]{MPR} for a proof of these properties).  We will call this metric the canonical metric associated to $a$ and $b$. Since rescaling the metric is a global conformal diffeomorphism, it will not affect our discussion of totally umbilical surfaces and hence we will use this $2$-parameter family of metric Lie groups as framework in the sequel.

The orthonormal reference $\{E_1,E_2,E_3\}$ is left-invariant and satisfies
\begin{align*}
 [E_1,E_2]&=0,\\
[E_2,E_3]&=(1-a)bE_1-(1-a)E_2,\\
[E_3,E_1]&=(1+a)E_1+(1+a)bE_2,
\end{align*}
so Koszul formula allows us to compute
\begin{equation}\label{eqn:nabla-nounim}
\begin{array}{lll}
\nabla_{E_1}E_1=(1+ a)E_3,&\nabla_{E_1}E_2=ab E_3,&\nabla_{E_1}E_3=-(1+a)E_1-ab E_2,\\
\nabla_{E_2}E_1=ab E_3,&\nabla_{E_2}E_2=(1-a)E_3,&\nabla_{E_2}E_3=-ab E_1-(1-a)E_2,\\
\nabla_{E_3}E_1=b E_2,&\nabla_{E_3}E_2=-b E_1,&\nabla_{E_3}E_3=0.
\end{array}
\end{equation}
Following a similar reasoning as in the unimodular case, we can work out the Riemannian curvature tensor $R$ of the metric Lie group.

\begin{lemma}\label{lemma:R-nouni}
In the previous situation,
\begin{align*}
R=[(1-a)^2(1+b^2)-b^2]R_1&+[(1+a)^2(1+b^2)-b^2]R_2\\
&+[(1-a^2)(1+b^2)-b^2]R_3,
\end{align*}
where, for $i\in\{1,2,3\}$ and $X,Y,Z\in\X(G)$, the tensor $R_i$ is given by
\begin{align*}
R_i(X,Y)Z&=\langle X,Z\rangle Y-\langle Y,Z\rangle X-\langle Z,E_i\rangle\langle X,E_i\rangle Y\\&\quad+\langle Z,E_i\rangle\langle Y,E_i\rangle X-\langle Y,E_i\rangle\langle X,Z\rangle E_i+\langle X,E_i\rangle\langle Y,Z\rangle E_i.
\end{align*}
\end{lemma}

\begin{remark}\label{rmk:nounim-examples}
We will briefly explain some particular cases which will appear later. If $a=0$, then $G$ has constant sectional curvature $-1$, so it is isometric to $\h^3$. If $a=1$, then $G$ is isometric to the space $\E(-4,b)$ whose isometry group has dimension $4$, and $E_2$ is a unit Killing vector field. It is interesting to observe that $\E(\kappa,\tau)$ for negative $\kappa$ admits unimodular and non-unimodular Lie group structures.
\end{remark}

\section{Totally umbilical surfaces in the unimodular case}\label{sec:unimodular}

Let us suppose that $\Sigma$ is a smooth surface isometrically immersed in a $3$-dimensional unimodular metric Lie group $G$. If $\Sigma$ is totally umbilical, then there exists a function $\lambda\in C^\infty(\Sigma)$ such that $AX=-\nabla_XN=\lambda X$ for any $X\in\X(\Sigma)$, where $N$ is a smooth unit normal vector field to the immersion and $A$ its associated Weingarten operator. We will also denote by $\{E_1,E_2,E_3\}\subset\X(G)$ a left-invariant orthonormal frame and $c_1,c_2,c_3\in\R$ satisfying~\eqref{eqn:lie-bracket-unimodular-frame}.

Let us consider a smooth parametrization $\phi:\Omega\to\Sigma$, where $\Omega\subset\R^2$ is an open domain and $\phi_u=\phi_*(\partial_u)$, $\phi_v=\phi_*(\partial_v)$ are the basic vector fields. On the one hand, by using the umbilicity condition, we can compute
\begin{equation}\label{eqn:tensor1}
R(\phi_u,\phi_v)N=\nabla_{\phi_u}\nabla_{\phi_v}N-\nabla_{\phi_v}\nabla_{\phi_u}N=\lambda_v \phi_u-\lambda_u\phi_v.
\end{equation}
On the other hand, let us write $N=\sum_{i=1}^3\nu_iE_i$, where the functions $\nu_i\in C^\infty(\Sigma)$ are given by $\nu_i=\langle N,E_i\rangle$ and satisfy $\nu_1^2+\nu_2^2+\nu_3^2=1$. They will be called \emph{angle functions} of the immersion. The unit normal field $N$ can be identified with the so-called left-invariant Gauss map 
\[N\equiv(\nu_1,\nu_2,\nu_3):\Sigma\to\s^2\subset\R^3.\] 

Now, expressing $\phi_u=\sum_{k=1}^3x_kE_k$, $\phi_v=\sum_{k=1}^3y_kE_k$, and using the tensors $R_i$ defined by Lemma~\ref{lemma:R-unimodular}, we obtain
\begin{equation}\label{eqn:tensor2}
R_i(\phi_u,\phi_v)N=\nu_i(y_i\phi_u-x_i\phi_v),\quad i\in\{1,2,3\},
\end{equation}
where the fact that $\phi_u$ and $\phi_v$ are orthogonal to $N$ is a handy condition. Thus, by comparing the coefficients in $\phi_u$ and $\phi_v$ in equations~\eqref{eqn:tensor1} and~\eqref{eqn:tensor2}, we get
\begin{align*}
\lambda_u&=(\mu_2\mu_3-c_1\mu_1)\nu_1x_1+(\mu_1\mu_3-c_2\mu_2)\nu_2x_2+(\mu_1\mu_2-c_3\mu_3)\nu_3x_3,\\
\lambda_v&=(\mu_2\mu_3-c_1\mu_1)\nu_1y_1+(\mu_1\mu_3-c_2\mu_2)\nu_2y_2+(\mu_1\mu_2-c_3\mu_3)\nu_3y_3.
\end{align*}

Finally, since $\lambda_u=\langle\phi_u,\nabla\lambda\rangle$ and $\lambda_v=\langle \phi_v,\nabla\lambda\rangle$, we reach the following expression for the gradient of $\lambda$.
\[\nabla\lambda=(\mu_2\mu_3-c_1\mu_1)\nu_1E^\top_1+(\mu_1\mu_3-c_2\mu_2)\nu_2E^\top_2+(\mu_1\mu_2-c_3\mu_3)\nu_3E^\top_3,\]
where $X^\top=X-\langle X,N\rangle N$ denotes the tangent component to $\Sigma$ of a vector field $X$. As $0=N^\top=\sum_{i=1}^3\nu_iE_i^\top$, we can simplify the last identity as
\begin{equation}\label{eqn:gradiente-unimodular}
\begin{aligned}
\nabla\lambda&=2\mu_2\mu_3\nu_1E^\top_1+2\mu_1\mu_3\nu_2E^\top_2+2\mu_1\mu_2\nu_3E_3^\top\\
&=2\mu_2(\mu_3-\mu_1)\nu_1E_1^\top+2\mu_1(\mu_3-\mu_2)\nu_2E_2^\top.
\end{aligned}
\end{equation}

Inspired by the ideas  in~\cite{ST09}, we will compute $[E_1^\top,E_2^\top](\lambda)$ in two different ways, which will give us an extra equation for the angle functions $\nu_i=\langle N,E_i\rangle$. First of all, we work out the gradients of the angle functions, which will be useful in the next computations.

\begin{lemma}\label{lemma:gradiente-angulo-unimodular}
The following equations hold:
\begin{align*}
\nabla\nu_1&=-\lambda E_1^\top-\mu_2\nu_3E_2^\top+\mu_3\nu_2E_3^\top,\\
\nabla\nu_2&=-\lambda E_2^\top+\mu_1\nu_3E_1^\top-\mu_3\nu_1E_3^\top,\\
\nabla\nu_3&=-\lambda E_3^\top+\mu_2\nu_1E_2^\top-\mu_1\nu_2E_1^\top.
\end{align*}
\end{lemma}

\begin{proof}
Observe that, for any $X\in\X(G)$ and $i\in\{1,2,3\}$, we can write using the umbilicity
\begin{align*}
\langle\nabla\nu_i,X\rangle=\langle\nabla\nu_i,X^\top\rangle&=X^\top(\langle E_i,N\rangle)=\langle\overline\nabla_{X^\top}E_i,N\rangle+\langle E_i,\nabla_{X^\top}N\rangle
\\&=\sum_{j=1}^3\langle X^\top,E_j^\top\rangle\langle\overline\nabla_{E_j}E_i,N\rangle-\lambda\langle E_i,X^\top\rangle.
\end{align*}
If we choose $X=E_k$ for $k\in\{1,2,3\}$ and take into account that 
\begin{equation}\label{eqn:scalarproduct}
\langle E_k,E_j^\top\rangle=\langle E_k^\top,E_j^\top\rangle=\delta_{jk}-\nu_j\nu_k,
\end{equation}
where $\delta_{jk}$ is the Kronecker delta, the statement follows.
\end{proof}

To compute $[E_1^\top,E_2^\top](\lambda)$, we will write it as $\langle[E_1^\top,E_2^\top],\nabla\lambda\rangle$ and calculate $[E_1^\top,E_2^\top]=\nabla_{E_1^\top}E_2^\top-\nabla_{E_2^\top}E_1^\top$. Now observe that, given $i,j\in\{1,2,3\}$,
\begin{align*}
\nabla_{E_i^\top}E_j^\top=\nabla_{E_i^\top}(E_j-\nu_jN)&=\left(\overline\nabla_{E_i^\top}E_j-E_i^\top(\nu_j)N-\nu_j\nabla_{E_i^\top}N\right)^\top\\
&=\lambda\nu_jE_i^\top+\sum_{k=1}^3\langle E_i^\top,E_k^\top\rangle\left(\overline\nabla_{E_k}{E_j}\right)^\top.
\end{align*}
Making $(i,j)=(1,2)$ or $(i,j)=(2,1)$ and use \eqref{eqn:nabla-unimodular} and \eqref{eqn:scalarproduct}, we get
\begin{align*}
\nabla_{E_2^\top}E_1^\top&=(\lambda\nu_1-\mu_3\nu_2\nu_3)E_2^\top-\mu_2(1-\nu_2^2)E_3^\top,
\\
\nabla_{E_1^\top}E_2^\top&=(\lambda\nu_2+\mu_3\nu_1\nu_3)E_1^\top+\mu_1(1-\nu_1^2)E_3^\top.
\end{align*}
By taking into account that $\nu_1E_1^\top+\nu_2E_2^\top+\nu_3E_3^\top=0$, we can simplify
\begin{equation}\label{eqn:corchete1-unimodular}
\begin{aligned}[]
[E_1^\top,E_2^\top]&=\nabla_{E_1^\top}E_2^\top-\nabla_{E_2^\top}E_1^\top\\
&=\lambda(\nu_2E_1^\top-\nu_1E_2^\top)+(\mu_1(1-\nu_1^2)+\mu_2(1-\nu_2^2)-\mu_3\nu_3^2)E_3^\top,
\end{aligned}
\end{equation}
so we finally obtain the desired expression from \eqref{eqn:gradiente-unimodular} and \eqref{eqn:corchete1-unimodular}:
\begin{align}
[E_1^\top,E_2^\top](\lambda)&=\langle[E_1^\top,E_2^\top],\nabla\lambda\rangle\notag\\
&=2\lambda\mu_3(\mu_2-\mu_1)\nu_1\nu_2-2\left(\mu_1(1-\nu_1^2)+\mu_2(1-\nu_2^2)-\mu_3\nu_3^2\right)\cdot\notag\\
&\quad\quad\quad\quad\quad\quad\quad\quad\quad\cdot\left(\mu_2(\mu_3-\mu_1)\nu_1^2\nu_3+\mu_1(\mu_3-\mu_2)\nu_2^2\nu_3\right).\label{eqn:corchete-lambda-unimodular}
\end{align}

As mentioned before, we will compute the bracket in another way; namely, we will compute $[E_1^\top,E_2^\top](\lambda)=E_1^\top(E_2^\top(\lambda))-E_2^\top(E_1^\top(\lambda))$. Using \eqref{eqn:gradiente-unimodular}, it is easy to check that
\begin{align*}
E_1^\top(\lambda)&=\langle E_1^\top,\nabla\lambda\rangle=2\mu_2(\mu_3-\mu_1)\nu_1(1-\nu_1^2)-2\mu_1(\mu_3-\mu_2)\nu_1\nu_2^2,\\
E_2^\top(\lambda)&=\langle E_2^\top,\nabla\lambda\rangle=-2\mu_2(\mu_3-\mu_1)\nu_1^2\nu_2+2\mu_1(\mu_3-\mu_2)\nu_2(1-\nu_2^2).
\end{align*}
Hence, we can apply Lemma~\ref{lemma:gradiente-angulo-unimodular} to take derivatives in these two last expressions and obtain
\begin{align}
 E_2^\top(E_1^\top(\lambda))&=2\mu_2(\mu_3-\mu_1)(1-3\nu_1^2)\langle E_2^\top,\nabla\nu_1\rangle\notag\\
&\quad\quad-2\mu_1(\mu_3-\mu_2)(\nu_2^2\langle E_2^\top,\nabla\nu_1\rangle+2\nu_1\nu_2\langle E_2^\top,\nabla\nu_2\rangle)\notag\\
&=2\mu_2(\mu_3-\mu_1)(1-3\nu_1^2)(\lambda\nu_1\nu_2-\mu_2(1-\nu_2^2)\nu_3-\mu_3\nu_2^2\nu_3)\notag\\
&\quad\quad-2\mu_1(\mu_3-\mu_2)\nu_2^2(\lambda\nu_1\nu_2-\mu_2(1-\nu_2^2)\nu_3-\mu_3\nu_2^2\nu_3)\notag\\
&\quad\quad+4\mu_1(\mu_3-\mu_2)\nu_1\nu_2(\lambda(1-\nu_2^2)-(\mu_3-\mu_1)\nu_1\nu_2\nu_3),\label{eqn-corchete21-unimodular}
\end{align}
\begin{align}
 E_1^\top(E_2^\top(\lambda))&=-2\mu_2(\mu_3-\mu_1)(2\nu_1\nu_2\langle E_1^\top,\nabla\nu_1\rangle+\nu_1^2\langle E_1^\top,\nabla\nu_2\rangle)\notag\\
&\quad\quad+2\mu_1(\mu_3-\mu_2)(1-3\nu_2^2)\langle E_1^\top,\nabla\nu_2\rangle\notag\\
&=4\mu_2(\mu_3-\mu_1)\nu_1\nu_2(\lambda(1-\nu_1^2)-(\mu_2-\mu_3)\nu_1\nu_2\nu_3)\notag\\
&\quad\quad-2\mu_2(\mu_3-\mu_1)\nu_1^2(\lambda\nu_1\nu_2+\mu_1(1-\nu_1^2)\nu_3+\mu_3\nu_1^2\nu_3)\notag\\
&\quad\quad+2\mu_1(\mu_3-\mu_2)(1-3\nu_2^2)(\lambda\nu_1\nu_2+\mu_1(1-\nu_1^2)\nu_3+\mu_3\nu_1^2\nu_3).\label{eqn-corchete22-unimodular}
\end{align}

\begin{lemma}\label{lemma:ecuacion2-unimodular}
The angle functions of a totally umbilical surface in $G$ satisfy
\begin{equation}\label{eqn:sistema2-unimodular}
\left\{\begin{array}{r}
\nu_1^2+\nu_2^2+\nu_3^2=1,\\
\beta_1\nu_1^2+\beta_2\nu_2^2+\beta_3\nu_3^2=0,
\end{array}\right.
\end{equation}
where $\beta_1,\beta_2,\beta_3$ are the real numbers defined by
\begin{align}
\beta_1&=\mu_2^2(\mu_1-\mu_3)+\mu_3^2(\mu_1-\mu_2),\notag\\
\beta_2&=\mu_3^2(\mu_2-\mu_1)+\mu_1^2(\mu_2-\mu_3),\label{eqn:beta-unimodular}\\
\beta_3&=\mu_1^2(\mu_3-\mu_2)+\mu_2^2(\mu_3-\mu_1),\notag
\end{align}
which depend only on the structure of $G$ and satisfy $\beta_1+\beta_2+\beta_3=0$. 
\end{lemma}

\begin{proof}
It suffices to substract~\eqref{eqn-corchete21-unimodular} from \eqref{eqn-corchete22-unimodular} and impose that the result is equal to \eqref{eqn:corchete-lambda-unimodular}.  The second equation in \eqref{eqn:sistema2-unimodular} follows from simplifying the resulting equality, and the first one from the fact that $N$ is unitary.
\end{proof}

As $\beta_1+\beta_2+\beta_3=0$, it is obvious that any $\nu_1,\nu_2,\nu_3\in\R$ such that $\nu_1^2=\nu_2^2=\nu_3^2=\frac{1}{3}$  satisfy the two equations in~\eqref{eqn:sistema2-unimodular}, so its set of solutions is non-empty. On the other hand, the conditions in~\eqref{eqn:sistema2-unimodular} say that the image of the left-invariant Gauss map $(\nu_1,\nu_2,\nu_3):\Sigma\to\s^2$ has dimension at most one, unless the system has rank $1$ as a linear system in the unknowns $\{\nu_1^2,\nu_2^2,\nu_3^2\}$ or, equivalently, when $\beta_1^2+\beta_2^2+\beta_3^2=0$.

It is easy to check from the expressions of the $\beta_i$ that they all vanish if and only if either $\mu_1=\mu_2=\mu_3$ (which means $c_1=c_2=c_3$) or two of the $\mu_i$ vanish (which means $c_1=c_2$, $c_3=0$ up to a permutation of indexes). We conclude that the system has rank $2$ except for $\R^3$, $\s^3$ or $\widetilde E(2)$ with the standard metric (i.e., except for those spaces with constant sectional curvature). As totally umbilical surfaces in these ambient manifolds are well-known, we will assume that $\beta_1^2+\beta_2^2+\beta_3^2\neq 0$ from now on, so the system~\eqref{eqn:sistema2-unimodular} can solved parametrically as
\begin{equation}\label{eqn:parametric-umbilical}
\begin{cases}\nu_1^2=\tfrac{1}{3}+(\beta_3-\beta_2)\eta,&\\
 \nu_2^2=\tfrac{1}{3}+(\beta_1-\beta_3)\eta,&\\
 \nu_3^2=\tfrac{1}{3}+(\beta_2-\beta_1)\eta,&
\end{cases}
\end{equation}
for a certain function $\eta\in C^\infty(\Sigma)$.

\begin{lemma}\label{lemma:ecuacion3-unimodular}
Under the hypothesis $\beta_1^2+\beta_2^2+\beta_3^2\neq 0$, the angle functions of a totally umbilical surface in $G$, with umbilicity function $\lambda$, satisfy
\begin{equation}\label{eqn:equation3-unimodular}
\lambda^2+\mu_2\mu_3\nu_1^2+\mu_1\mu_3\nu_2^2+\mu_1\mu_2\nu_3^2=0.
\end{equation}
\end{lemma}

\begin{proof}
As the image of the left-invariant Gauss map has dimension at most one, given any point $p\in\Sigma$, there exists a tangent direction $u\in T_p\Sigma$ such that $dN_p(u)=0$ (i.e., $\langle u,\nabla\nu_i\rangle=0$ for $i\in\{1,2,3\}$). Writing $u=\sum_{i=1}^3b_iE_i$, we deduce that
\begin{align*}
 \nabla_uN&=\sum_{i=1}^3\langle u,\nabla\nu_i\rangle E_i+\sum_{i=1}^3\nabla_{\gamma'}E_i=\sum_{i,j=1}^3\nu_ib_j\overline\nabla_{E_j}E_i\\
 &=(\mu_2\nu_3b_2-\mu_3\nu_2b_3)E_1+(\mu_3\nu_1b_3-\mu_1\nu_3b_1)E_2+(\mu_1\nu_2b_1-\mu_2\nu_1b_2)E_3,
\end{align*}
where everything is computed at $p$. Identifying the corresponding coefficients in the umbilicity condition $\nabla_uN=-\lambda u=-\sum_{i=1}^3\lambda b_iE_i$, we obtain three equations which, together with the fact that $0=\langle u,N\rangle=\sum_{i=1}^3b_i\nu_i$, can be written in matrix-form as
\begin{equation}\label{eqn:sistema-nivel-unimodular}
\left(\begin{matrix}\lambda&\mu_2\nu_3&-\mu_3\nu_2\\-\mu_1\nu_3&\lambda&\mu_3\nu_1\\\mu_1\nu_2&-\mu_2\nu_1&\lambda\\\nu_1&\nu_2&\nu_3\end{matrix}\right)
\left(\begin{matrix}b_1\\b_2\\b_3\end{matrix}\right)=\left(\begin{matrix}0\\0\\0\\0\end{matrix}\right).\end{equation}
As this linear system has a non-trivial solution, the four $3$-minors of the coefficient matrix must vanish, giving rise to the following four equations:
\begin{align*}
\nu_1(\lambda^2+\mu_2\mu_3\nu_1^2+\mu_1\mu_3\nu_2^2+\mu_1\mu_2\nu_3^2)=0,\\
\nu_2(\lambda^2+\mu_2\mu_3\nu_1^2+\mu_1\mu_3\nu_2^2+\mu_1\mu_2\nu_3^2)=0,\\
\nu_3(\lambda^2+\mu_2\mu_3\nu_1^2+\mu_1\mu_3\nu_2^2+\mu_1\mu_2\nu_3^2)=0,\\
\lambda(\lambda^2+\mu_2\mu_3\nu_1^2+\mu_1\mu_3\nu_2^2+\mu_1\mu_2\nu_3^2)=0.
\end{align*}
As $\nu_1^2+\nu_2^2+\nu_3^2=1$, equation \eqref{eqn:equation3-unimodular} holds at the arbitrary point $p\in\Sigma$.
\end{proof}

Let us now put together \eqref{eqn:sistema2-unimodular} and \eqref{eqn:equation3-unimodular} to get
\begin{equation}\label{eqn:sistema3-unimodular}
\left\{\begin{array}{r}
\nu_1^2+\nu_2^2+\nu_3^2=1,\\
\beta_1\nu_1^2+\beta_2\nu_2^2+\beta_3\nu_3^2=0,\\
\mu_2\mu_3\nu_1^2+\mu_1\mu_3\nu_2^2+\mu_1\mu_2\nu_3^2=-\lambda^2.
\end{array}\right.
\end{equation}
As a linear system with unknowns $\{\nu_1^2,\nu_2^2,\nu_3^2\}$, the determinant of the coefficient matrix is given by
\begin{equation}\label{eqn:delta-unimodular}
\Delta=(\mu_1-\mu_2)(\mu_2-\mu_3)(\mu_3-\mu_1)(\mu_1\mu_2+\mu_2\mu_3+\mu_1\mu_3).
\end{equation}
We will distinguish cases depending on whether or not this determinant vanishes (i.e., whether or not the system \eqref{eqn:sistema3-unimodular} is degenerate). 

It is convenient to begin by discussing totally umbilical surfaces with $\lambda$ constant or, equivalently, when $\|\nabla\lambda\|=0$. The following formula for the squared norm of the gradient of $\lambda$ can be deduced directly from~\eqref{eqn:gradiente-unimodular} and will be useful for this purpose:
\begin{equation}\label{eqn:nablalambda-unimodular}
\begin{array}{rl}\|\nabla\lambda\|^2&=4\mu_2^2(\mu_3-\mu_1)^2\nu_1^2(1-\nu_1^2)+4\mu_1^2(\mu_3-\mu_2)^2\nu_2^2(1-\nu_2^2)\\
&\quad\quad-8\mu_1\mu_2(\mu_3-\mu_1)(\mu_3-\mu_2)\nu_1^2\nu_2^2.
\end{array}
\end{equation}
Next result generalizes and gives an alternative proof of~\cite[Theorem 7.2]{Tsu96}, where totally geodesic surfaces in unimodular metric Lie groups are classified.

\begin{proposition}\label{prop:lambda-constant-unimodular}
Let us suppose that $\beta_1^2+\beta_2^2+\beta_3^2\neq 0$ and let $\Sigma$ be a totally umbilical surface in $G$ with constant umbilicity function. Then:
\begin{itemize}
 \item[a)] $\Sigma$ is totally geodesic and $\Delta\neq0$.
 \item[b)] If we suppose that $c_3\leq c_2\leq c_1$, then $c_3<0<c_1$, $c_2=c_1+c_3$, and $\Sigma$ is an integral surface of one of the distribution spanned by $\{\sqrt{c_1}E_1+\sqrt{-c_3}E_3,E_2\}$ or $\{\sqrt{c_1}E_1-\sqrt{-c_3}E_3,E_2\}$.
\end{itemize}
\end{proposition}

\begin{remark}
The distributions in the statement can be easily shown to be integrable (in fact, they span Lie subalgebras of the Lie algebra of $G$). It is also easy to check that any integral surface of the distribution is a totally geodesic surface.
\end{remark}

\begin{proof}
After substituting the parametric expressions given by~\eqref{eqn:parametric-umbilical} in~\eqref{eqn:nablalambda-unimodular}, the condition $\|\nabla\lambda\|^2=0$ can be written in terms of $\eta\in C^\infty(\Sigma)$ as a second-order equation $a_2\eta^2+a_1\eta+a_0=0$, where
\begin{align*}
 a_0&=\tfrac{8}{9}(\mu_1^2\mu_2^2+\mu_2^2\mu_3^2+\mu_3^2\mu_1^2-\mu_1\mu_2\mu_3(\mu_1+\mu_2+\mu_3)),\\
 a_1&=\tfrac{4}{3}(\mu_1-\mu_2)(\mu_2-\mu_3)(\mu_3-\mu_1)\cdot\\
&\quad\quad\cdot(4\mu_1\mu_2\mu_3(\mu_1+\mu_2+\mu_3)-\mu_1^2\mu_2^2-\mu_2^2\mu_3^2-\mu_3^2\mu_1^2),\\
 a_2&=-4(\mu_1-\mu_2)^2(\mu_2-\mu_3)^2(\mu_3-\mu_1)^2(\mu_1\mu_2+\mu_2\mu_3+\mu_3\mu_1)^2.
\end{align*}
As $a_0,a_1,a_2$ are constants depending on $\mu_1,\mu_2,\mu_3$, we deduce that $\eta$ is constant, so~\eqref{eqn:parametric-umbilical} implies that the left-invariant Gauss map is also constant. Since $a_2=-4\Delta^2$, we will distinguish the cases $\Delta=0$ and $\Delta\neq 0$.

In view of~\eqref{eqn:delta-unimodular}, the condition $\Delta=0$ gives rise to two subcases:
\begin{itemize}
 \item Two of the $\mu_i$ are equal. If we suppose that $\mu_1=\mu_2$ with no loss of generality, it turns out that $a_1=a_2=0$ and $a_0=\frac{8}{9}\mu_2^2(\mu_2-\mu_3)$, from where we deduce that $a_0$ must be equal to $0$. Thus, either $\mu_1=\mu_2=0$ or $\mu_1=\mu_2=\mu_3$, so $\beta_1^2+\beta_2^2+\beta_3^2=0$ in any case.
\item If $\mu_1\mu_2+\mu_2\mu_3+\mu_1\mu_3=0$ and no two of the $\mu_i$ are equal, then $a_1\neq0$ so we can solve for $\eta$ and obtain the unique solution
\begin{equation}\label{eqn:eta1-unimodular}
\eta=\frac{1}{3(\mu_1-\mu_2)(\mu_2-\mu_3)(\mu_3-\mu_1)},
\end{equation}
which allows us to substitute in \eqref{eqn:parametric-umbilical} and obtain
\begin{align}
 \nu_1^2=\frac{\mu_1^2}{(\mu_1-\mu_2)(\mu_1-\mu_3)}\notag\\
 \nu_2^2=\frac{\mu_2^2}{(\mu_2-\mu_1)(\mu_2-\mu_3)}\label{eqn:nu-unimodular-deltacero}\\
 \nu_3^2=\frac{\mu_3^2}{(\mu_3-\mu_1)(\mu_3-\mu_2)}\notag
\end{align}
We can suppose, without loss of generality that $\mu_1<\mu_2<\mu_3$, which gives $\nu_2^2<0$ unless $\mu_2=0$. Moreover, the conditions $\mu_2=0$ and $\mu_1\mu_2+\mu_2\mu_3+\mu_1\mu_3=0$ imply that $\mu_1=0$ or $\mu_3=0$. In any case, it follows that $\beta_1^2+\beta_2^2+\beta_3^2=0$.
\end{itemize}
Let us now suppose that $\Delta\neq0$, so the equation $a_2\eta^2+a_1\eta+a_0=0$ has two solutions. The first one is given by~\eqref{eqn:eta1-unimodular}, so~\eqref{eqn:nu-unimodular-deltacero} is also satisfied. As in the discussion above, $\mu_1<\mu_2<\mu_3$ implies $\mu_2=0$ and $\nu_2=0$. The third equation in~\eqref{eqn:sistema3-unimodular} now yields $\lambda=0$, so $\Sigma$ is totally geodesic. Moreover, from $\mu_2=0$ we get $\mu_1=2c_3$ and $\mu_3=2c_1$, so~\eqref{eqn:nu-unimodular-deltacero} now reads as $\nu_1^2=\frac{-c_3}{c_1-c_3}$, $\nu_2=0$ and $\nu_3^2=\frac{c_1}{c_1-c_3}$. Depending on the choice of signs, a basis of the tangent bundle of $\Sigma$ is given either by $\{\sqrt{c_1}E_1+\sqrt{-c_3}E_3,E_2\}$ or by $\{\sqrt{c_1}E_1-\sqrt{-c_3}E_3,E_2\}$.

The other solution of $a_2\eta^2+a_1\eta+a_0=0$ is given by
\[\eta=\frac{-2(\mu_1^2\mu_2^2+\mu_1^2\mu_2^2+\mu_1^2\mu_2^2-\mu_1\mu_2\mu_3(\mu_1+\mu_2+\mu_3))}{3(\mu_1-\mu_2)(\mu_2-\mu_3)(\mu_3-\mu_1)(\mu_1\mu_2+\mu_2\mu_3+\mu_1\mu_3)^2}.\]
By substituting this value in~\eqref{eqn:parametric-umbilical}, we get
\begin{align}
\nu_1^2&=\frac{\beta_2\beta_3}{(\mu_1-\mu_2)(\mu_1-\mu_3)(\mu_1\mu_2+\mu_2\mu_3+\mu_1\mu_3)^2},\notag\\
 \nu_2^2&=\frac{\beta_1\beta_3}{(\mu_2-\mu_1)(\mu_2-\mu_3)(\mu_1\mu_2+\mu_2\mu_3+\mu_1\mu_3)^2},\label{eqn:nu-unimodular-a}\\
 \nu_3^2&=\frac{\beta_1\beta_2}{(\mu_3-\mu_1)(\mu_3-\mu_1)(\mu_1\mu_2+\mu_2\mu_3+\mu_1\mu_3)^2}.\notag
\end{align}
Since $\mu_1<\mu_2<\mu_3$, we deduce from~\eqref{eqn:beta-unimodular} that $\beta_1\leq 0$ and $\beta_3\geq 0$. The condition $\nu_1^2\geq 0$ in the first equation of~\eqref{eqn:nu-unimodular-a} yields $\beta_2\geq0$ whereas $\nu_3^2\geq0$ in the third one gives $\beta_2\leq0$. Thus, $\beta_2=0$ and $\nu_1=\nu_3=0$. In other words, the surface is orthogonal to the vector field $E_2$. Hence its tangent plane is globally generated by $\{E_1,E_3\}$ and its normal vectof field is $N=E_2$. In view of~\eqref{eqn:nabla-unimodular}, the condition $\nabla_{E_1}N=-\lambda E_1$ gives $\lambda=0$ and $\mu_1=0$. Likewise, the condition $\nabla_{E_3}N=-\lambda E_3$ gives $\mu_3=0$ so $\beta_1=\beta_2=\beta_3=0$, contradicting the hypothesis in the statement.
\end{proof}

As a consequence of Proposition~\ref{prop:lambda-constant-unimodular}, we obtain the following generalization of the non-existence theorem of totally umbilical surfaces in $\mathbb E(\kappa,\tau)$-spaces for $\tau\neq 0$ and $\kappa\neq 4\tau^2$ given in~\cite{ST09} (see also Remark~\ref{rmk:unimodular-examples}). 

\begin{corollary}\label{coro:deltacero}
If $\Delta=0$ and $\beta_1^2+\beta_2^2+\beta_3^2\neq 0$, then $G$ does not admit  totally umbilical surfaces.
\end{corollary}

\begin{proof}
The condition $\Delta=0$ is satisfied because either two of the $\mu_i$ are equal or because $\mu_1\mu_2+\mu_2\mu_3+\mu_1\mu_3=0$ so we will analyse both subcases:
\begin{itemize}
 \item If $\mu_1= \mu_2$, then the linear system~\eqref{eqn:sistema3-unimodular} is compatible if and only if $\lambda^2=\frac{1}{3}\mu_1(\mu_1+2\mu_3)$ so $\lambda$ is constant.
 \item If $\mu_1\mu_2+\mu_1\mu_3+\mu_2\mu_3=0$ and no two of the $\mu_i$ are equal, then the system \eqref{eqn:sistema3-unimodular} is compatible if and only if $\lambda^2(\mu_1+\mu_2+\mu_3)=0$ (to see this, it suffices to realize that, under the condition $\mu_1\mu_2+\mu_1\mu_3+\mu_2\mu_3=0$, the second row in \eqref{eqn:sistema3-unimodular} is equal to $-2(\mu_1+\mu_2+\mu_3)$ times the third one). Finally, we will prove that $\mu_1+\mu_2+\mu_3\neq0$ to conclude that $\lambda=0$. Indeed, $\mu_1+\mu_2+\mu_3=0$ implies $0=\mu_1\mu_2+\mu_3(\mu_1+\mu_2)=\mu_1\mu_2-(\mu_1+\mu_2)^2$, so $\mu_1^2+\mu_1\mu_2+\mu_2^2=0$ from where $\mu_1=\mu_2=\mu_3=0$ and $\beta_1^2+\beta_2^2+\beta_3^2=0$.
\end{itemize}
In both cases, $\lambda$ is constant and $\Delta=0$, contradicting Proposition~\ref{prop:lambda-constant-unimodular}.
\end{proof}

We will now suppose that $\Delta\neq 0$. As the coefficient matrix of the linear system~\eqref{eqn:sistema3-unimodular} is invertible in this case, we can solve for $\nu_1^2,\nu_2^2,\nu_3^2$ and get
\begin{equation}\label{eqn:nu-unimodular-deltanocero}
\begin{aligned}
\nu_1^2&=\frac{(\beta_3-\beta_2)}{\Delta}\lambda^2+\frac{\mu_1^2}{(\mu_1-\mu_2)(\mu_1-\mu_3)},\\
 \nu_2^2&=\frac{(\beta_1-\beta_3)}{\Delta}\lambda^2+\frac{\mu_2^2}{(\mu_2-\mu_1)(\mu_2-\mu_3)},\\
 \nu_3^2&=\frac{(\beta_2-\beta_1)}{\Delta}\lambda^2+\frac{\mu_3^2}{(\mu_3-\mu_1)(\mu_3-\mu_2)}.
\end{aligned}
\end{equation}

\begin{remark}
Plugging~\eqref{eqn:nu-unimodular-deltanocero} into~\eqref{eqn:nablalambda-unimodular}, we obtain
\begin{equation}\label{eqn:lambda-buena}
\|\nabla\lambda\|^2=4\lambda^2\left(a-\lambda^2\right),\quad a=-\frac{\mu_1^2\mu_2^2+\mu_2^2\mu_3^2+\mu_1^2\mu_3^2}{\mu_1\mu_2+\mu_2\mu_3+\mu_1\mu_3}.
\end{equation}
Observe that when $a<0$ (equivalently, when $G$ has positive scalar curvature, see equation~\eqref{eqn:rho-unimodular}), the identity~\eqref{eqn:lambda-buena} guarantees the non-existence of totally umbilical surfaces in $G$. On the other hand, $a=0$ implies $\Delta=0$ so totally umbilical surfaces do not exist for zero scalar curvature, either. Note that the case $\beta_1^2+\beta_2^2+\beta_3^2=0$ is not considered here.
\end{remark}

\begin{theorem}[The unimodular case]\label{thm:unimodular}
Let $G$ be an unimodular metric Lie group with structure constants $c_3\leq c_2\leq c_1$.
\begin{enumerate}
 \item If $c_1=c_2=c_3$ or $c_1=c_2$, $c_3=0$, then $G$ has constant sectional curvature and it is isometric to $\s^3$ or $\R^3$.
 \item If $c_3<0<c_1$ and $c_2=c_1+c_3$, then we distinguish two subcases:
\begin{itemize}
 \item If $c_2\neq 0$, then any totally umbilical surface in $G$ belongs to one of the two families of totally geodesic surfaces described in Proposition~\ref{prop:lambda-constant-unimodular}.
 \item If $c_2=0$, then $G=\Sol$ endowed with a left-invariant metric homothetical to the standard one. Any totally umbilical surface in $G$ is either totally geodesic (see Proposition~\ref{prop:lambda-constant-unimodular}), or it is, up to ambient isometries, the surface described in~\cite{ST09}.
\end{itemize}
 \item In the rest of cases, $G$ does not admit totally umbilical surfaces.
\end{enumerate}
\end{theorem}

\begin{proof}
We can suppose that $\beta_1^2+\beta_2^2+\beta_3^2\neq 0$ since, otherwise, we lie in item (1). By Propositions~\ref{prop:lambda-constant-unimodular} and Corollary~\ref{coro:deltacero}, it suffices to show that if $G$ admits a totally umbilical surface $\Sigma$ with a non-constant umbilicity function $\lambda$ (in particular $\Delta\neq0$), then $G$ is the $\Sol$ group with a metric homothetical to the standard one, and the result will follow from~\cite{ST09}.

Let us write $\nabla\lambda=\sum_{i=1}^3a_iE_i$ for some $a_1,a_2,a_3\in C^\infty(\Sigma)$. From equation~\eqref{eqn:gradiente-unimodular}, we can compute
\begin{align*}
 a_1&=(\lambda^2+\mu_2\mu_3)\nu_1,&
 a_2&=(\lambda^2+\mu_1\mu_3)\nu_2,&
 a_3&=(\lambda^2+\mu_1\mu_2)\nu_3.
\end{align*}
where a unit normal vector field $\Sigma$ is $N=\sum_{i=1}^3\nu_iE_i$. We will consider the vector field $X=\nabla\lambda\wedge N$ and write $X=\sum_{i=1}^3 b_iE_i$ for some functions $b_1,b_2,b_3\in C^\infty(\Sigma)$. It follows that
\begin{align}
b_1&=a_2\nu_3-a_3\nu_2={\mu_1(\mu_2-\mu_3)}\nu_2\nu_3,\notag\\
b_2&=a_3\nu_1-a_1\nu_3={\mu_2(\mu_3-\mu_1)}\nu_1\nu_3,\label{eqn:bi-unimodular}\\
b_3&=a_1\nu_2-a_2\nu_1={\mu_3(\mu_1-\mu_2)}\nu_1\nu_2.\notag
\end{align}
From~\eqref{eqn:nu-unimodular-deltanocero} we deduce that $dN_p(X_p)=0$ for all $p\in\Sigma$, where $N=(\nu_1,\nu_2,\nu_3)$ denotes the left-invariant Gauss-map. As in the argument of the proof of Lemma~\ref{lemma:ecuacion3-unimodular}, $(b_1,b_2,b_3)$ satisfies the linear system~\eqref{eqn:sistema-nivel-unimodular}. Substituting~\eqref{eqn:bi-unimodular} in~\eqref{eqn:sistema-nivel-unimodular} and simplifying we get the following three equations:
\begin{align*}
\lambda\nu_2\nu_3&=\frac{\mu_3^2(\mu_1-\mu_2)\nu_2^2+\mu_2^2 (\mu_1-\mu_3) \nu_3^2}{\mu_1(\mu_2-\mu_3)}\nu_1,\\
\lambda\nu_1\nu_3&=\frac{\mu_3^2(\mu_1-\mu_2)\nu_1^2+\mu_1^2 (\mu_3-\mu_2) \nu_3^2}{\mu_2(\mu_1-\mu_3)}\nu_2,\\
\lambda\nu_1\nu_2&=\frac{\mu_2^2(\mu_3-\mu_1)\nu_1^2+\mu_1^2 (\mu_3-\mu_2) \nu_3^2}{\mu_3(\mu_1-\mu_2)}\nu_3.
\end{align*}
We will only make use of the first equation. Squaring both sides and susbtituting~\eqref{eqn:nu-unimodular-deltanocero}, we get that $\lambda$ satisfies a degree $6$ polynomial equation $c_6\lambda^6+c_4\lambda^4+c_2\lambda^2=0$, where the constants $c_2$, $c_4$ and $c_6$ are given by
\begin{align*}
 c_2&=-\mu_1^2\mu_2^2\mu_3^2(\mu_1\mu_2+\mu_2\mu_3+\mu_1\mu_3)^3,\\
 c_4&=\mu_1\mu_2\mu_3\,Q(\mu_1,\mu_2,\mu_3),\\
 c_6&=(\mu_1+\mu_2+\mu_3)\bigl(\mu_1^2\mu_2^2+\mu_2^2\mu_3^2+\mu_1^2\mu_3^2-2\mu_1\mu_2\mu_3(\mu_1+\mu_2+\mu_3)\bigr)\cdot\\
&\quad\quad\cdot\bigl(\mu_1\mu_2(\mu_1-\mu_2)+\mu_1\mu_3(\mu_1-\mu_3)+\mu_2\mu_3(\mu_2+\mu_3)+\mu_1\mu_2\mu_3\bigr),
\end{align*}
and $Q(\mu_1,\mu_2,\mu_3)$ is a certain polynomial expression. In particular, the fact that $\lambda$ is not constant implies that $c_2=c_4=c_6=0$. Since $\Delta\neq 0$, the condition $c_2=0$ implies $\mu_1=0$ or $\mu_2=0$ or $\mu_3=0$, so automatically $c_4=0$. Moreover, no two of the $\mu_i$ vanish simultaneously, since it would also lead to $\Delta=0$. It is easy to check from $c_6=0$ that if one of the $\mu_i$ vanishes then the sum of the other two also vanish, so $G$ is the $\Sol$ group with a left-invariant metric homothetical to the standard one.
\end{proof}

\section{The non-unimodular case}\label{sec:nounimodular}
Let us consider $\Sigma$ a totally umbilical surface in a non-unimodular metric Lie group $G$. Up to a homothety in the metric, $G$ may be considered to be the semidirect product $\R^2\ltimes_{A(a,b)}\R$, where $A(a,b)$ is the $2\times 2$ real matrix given by~\eqref{eqn:A-nonuni} for some $a,b\geq 0$. In this setting, $\{E_1,E_2,E_3\}$ will stand for the left-invariant global orthonormal frame defined in Section~\ref{sec:nouni-intro} and $N=(\nu_1,\nu_2,\nu_3):\Sigma\to\s^2$ will represent the left-invariant Gauss map of $\Sigma$, where $\nu_i=\langle E_i,N\rangle$, $i\in\{1,2,3\}$, are the corresponding angle functions. We will denote by $\lambda\in C^\infty(\Sigma)$ the umbilicity function associated to $\Sigma$. 

Let $\phi:\Omega\subset\R^2\to\Sigma$ be a local parametrization. On the one hand, observe that equation~\eqref{eqn:tensor1} also holds in the non-unimodular case and, on the other hand, we can calculate $R(\phi_u,\phi_v)N$ by using Lemma~\eqref{lemma:R-nouni}. By comparing the terms in $\phi_u$ and $\phi_v$ with those in~\eqref{eqn:tensor1}, we get that the gradient of $\lambda$ is given by
\begin{equation}\label{eqn:nounim-nablalambda}
\nabla\lambda=2a(1+b^2)\bigl((a-1)\nu_1E_1^\top+(a+1)\nu_2E_2^\top\bigr).
\end{equation}

\begin{lemma}\label{lemma:angle-gradients}
The gradients of the angle functions are given by
\begin{align*}
\nabla\nu_1&=((1+a)\nu_3-\lambda)E_1^\top+ab\nu_3E_2^\top+b\nu_2E_3^\top,\\
\nabla\nu_2&=ab\nu_3E_1^\top+((1-a)\nu_3-\lambda)E_2^\top-b\nu_1E_3^\top,\\
\nabla\nu_3&=-((1+a)\nu_1+ab\nu_2)E_1^\top-(ab\nu_1+(1-a)\nu_2)E_2^\top-\lambda E_3^\top.
\end{align*}
\end{lemma}

\begin{proof}
Similar to the proof of Lemma~\ref{lemma:gradiente-angulo-unimodular}.
\end{proof}

\begin{lemma}\label{lemma:nounim-equation1}
Let us suppose that $a\neq 0$. Then:
\begin{equation}\label{eqn:nounim-equation1}
\left((a+1)(a+2)\nu_2^2-(a-1)(a-2)\nu_1^2-2a\right)b+2(a^2-1)\nu_1\nu_2=0.
\end{equation}
\end{lemma}

\begin{proof}
We can argue as in the proof of Lemma~\ref{lemma:ecuacion2-unimodular}. Computations are long but straightforward, so we will omit some of the details. On the one hand,
\begin{align*}
\nabla_{E_1^\top}E_2^\top&=\bigr(b\nu_1\nu_3+\lambda\nu_2\bigl)E_1^\top+\bigl(ab(1-\nu_1^2)-(1-a)\nu_1\nu_2\bigr)E_3^\top,\\
\nabla_{E_2^\top}E_1^\top&=\bigl(-b\nu_2\nu_3+\lambda\nu_1\bigr)E_2^\top+\bigl(ab(1-\nu_2^2)-(1+a)\nu_1\nu_2\bigr)E_3^\top,
\end{align*}
and subtracting $\nabla_{E_2^\top}E_1^\top$ from $\nabla_{E_1^\top}E_2^\top$, we reach
\begin{equation}\label{eqn:nounim-bracket1}
[E_1^\top,E_2^\top]=\lambda\nu_2E_1^\top-\lambda\nu_1E_2^\top+\bigr(ab(\nu_2^2-\nu_1^2)+2a\nu_1\nu_2-b\nu_3^2\bigl)E_3^\top.
\end{equation}
On the other hand, from \eqref{eqn:nounim-nablalambda} we deduce that
\begin{align*}
 E_1^\top(\lambda)&=2a(1+b^2)\bigl((a-1)\nu_1(1-\nu_1^2)-(a+1)\nu_1\nu_2^2\bigr),\\
 E_2^\top(\lambda)&=2a(1+b^2)\bigl((a+1)\nu_2(1-\nu_2^2)-(a-1)\nu_1^2\nu_2\bigr).
\end{align*}
Using these equalities and Lemma~\ref{lemma:angle-gradients} it is easy to compute $E_1^\top(E_2^\top(\lambda))$ and $E_2^\top(E_1^\top(\lambda))$. Subtracting these terms and imposing that the result must be equal to $\langle[E_1^\top,E_2^\top],\nabla\lambda\rangle$, where $[E_1^\top,E_2^\top]$ is given by~\eqref{eqn:nounim-bracket1}, we finally reach the identity
\[a\nu_3\Bigl(\left((a+1)(a+2)\nu_2^2-(a-1)(a-2)\nu_1^2-2a\right)b+2(a^2-1)\nu_1\nu_2\Bigr)=0,\]
where we can simplify the factor $a\neq0$. Moreover, if we repeat the argument above by using the pairs $(E_1^\top,E_3^\top)$ and $(E_2^\top,E_3^\top)$, rather than $(E_1^\top,E_2^\top)$, we obtain, after simplifying by $a\neq 0,$ the following two identities, respectively:
\begin{align*}
\nu_2\Bigl(\left((a+1)(a+2)\nu_2^2-(a-1)(a-2)\nu_1^2-2a\right)b+2(a^2-1)\nu_1\nu_2\Bigr)&=0,\\
\nu_1\Bigl(\left((a+1)(a+2)\nu_2^2-(a-1)(a-2)\nu_1^2-2a\right)b+2(a^2-1)\nu_1\nu_2\Bigr)&=0.
\end{align*}
Since $\nu_1$, $\nu_2$ and $\nu_3$ do not vanish simultaneously, the equality in the statement follows.
\end{proof}

Now observe that $a=0$ implies that the group $G$ is isometric to $\h^3$ since it has constant sectional curvature $-1$. On the other hand, except for the case $a=1,b=0$, which corresponds to the homogeneous product space $\h^2\times\R$,  Lemma~\ref{lemma:nounim-equation1} implies that the image of the left-invariant Gauss map $N=(\nu_1,\nu_2,\nu_3):\Sigma\to\s^2$ is contained in a curve. Note that if $a=1$ and $b$ is arbitrary, then $G$ is isometric to $\E(-4,b)$ so this case can be also ruled out (see Remark~\ref{rmk:nounim-examples}).

\begin{lemma}\label{lemma:nounim-equation2}
If $a\not\in\{0,1\}$, then the following identities hold:
\begin{align}
\lambda ^2-2\nu_3\lambda -a b^2\nu_1^2+a b^2\nu_2^2-a^2(1+b^2)\nu_3^2+\nu_3^2+2 a b \nu_1\nu_2&=0,\label{eqn:nounim-equation2}\\
4b\nu_1^2\nu_2^2+((1-a)\nu_1^2-(1+a)\nu_2^2+2a)ab\nu_3^2&=0.\label{eqn:nounim-equation3}
\end{align}
\end{lemma}

\begin{proof}
Since the image of the left-invariant Gauss map $N$ of $\Sigma$ has dimension at most $1$, we can follow the argument in the proof of Lemma~\ref{lemma:ecuacion3-unimodular}. Similarly to the deduction of~\eqref{eqn:sistema-nivel-unimodular}, if we take a direction $u=\sum_{i=1}^3b_iE_i$ tangent to $\Sigma$ at some point $p$, and such that $dN_p(u)=0$, then
\begin{equation}\label{eqn:sistema-nivel-nounimodular}
\left(\begin{matrix}-(1+a)\nu_3+\lambda&-ab\nu_3&-b\nu_2\\-ab\nu_3&-(1-a)\nu_3+\lambda&b\nu_1\\(1+a)\nu_1+ab\nu_2&ab\nu_1+(1-a)\nu_2&\lambda\\\nu_1&\nu_2&\nu_3\end{matrix}\right)
\left(\begin{matrix}b_1\\b_2\\b_3\end{matrix}\right)=\left(\begin{matrix}0\\0\\0\\0\end{matrix}\right).\end{equation}
The minors of order $3$ of this $4\times 3$ real matrix must vanish but this can be shown to happen if and only if equation~\eqref{eqn:nounim-equation2} holds.

In order to obtain~\eqref{eqn:nounim-equation3}, let us consider the vector field
\[X=\frac{\nabla\lambda\wedge N}{2a(1+b^2)}=(1+a)\nu_2\nu_3E_1+(1-a)\nu_1\nu_3E_2-2\nu_1\nu_2E_3,\]
which has been computed by expressing $\nabla\lambda$, given in~\eqref{eqn:nounim-nablalambda}, in terms of the orthonormal basis $\{E_1,E_2,E_3\}$. From~\eqref{eqn:nounim-equation2}, we deduce that $\lambda$ is constant along the curves where $N$ is constant. This means that $\nabla\lambda$ is orthogonal to the level curves of $N$ so $X$ is tangent to $\Sigma$ and $dN_p(X_p)=0$ for all $p\in\Sigma$. Consequently, we get the following three equations by  plugging  $b_1=(1+a)\nu_2\nu_3$, $b_2=(1-a)\nu_1\nu_3$ and $b_3=-2\nu_1\nu_2$ in~\eqref{eqn:sistema-nivel-nounimodular}:
\begin{equation}\label{eqn:nounim-extra}
\begin{aligned}
(a+1)\lambda\nu_2\nu_3&=(a+1)^2\nu_2\nu_3^2-b\nu_1(2\nu_2^2+a(a-1)\nu_3^2),\\
(a-1)\lambda\nu_1\nu_3&=-(a-1)^2\nu_1\nu_3^2-b\nu_2(2\nu_1^2+a(a+1)\nu_3^2),\\
2\lambda\nu_1\nu_2&=(-ab(a-1)\nu_1^2+2(1+a^2)\nu_1\nu_2+ab(a+1)\nu_2^2)\nu_3.
\end{aligned}
\end{equation}
Let us now multiply the first one by $\nu_1$ and the second one by $\nu_2$, and isolate $\lambda\nu_1\nu_2\nu_3$ in both of them. We then use the following substitution in the right-hand-sides of the two resulting identities to transform the terms with $\nu_1\nu_2\nu_3^2$:
\[\nu_1\nu_2=\frac{\left((a+1)(a+2)\nu_2^2-(a-1)(a-2)\nu_1^2-2a\right)b}{2(1-a^2)},\]
This substitution follows from~\eqref{eqn:nounim-equation1} and the fact that $a^2\neq 1$. We finally get~\eqref{eqn:nounim-equation3} by identifying the terms $\lambda\nu_1\nu_2\nu_3$. We remark that using the third equality in~\eqref{eqn:nounim-extra} leads to a equation equivalent to~\eqref{eqn:nounim-equation3}.
\end{proof}

\begin{theorem}[The non-unimodular case]\label{thm:nounimodular}
Let $G=\R^2\ltimes_{A(a,b)}\R$ be the non-unimodular metric Lie group obtained for some constants $a,b\geq 0$.
\begin{enumerate}
 \item If $a=0$, then $G$ is isometric to the hyperbolic space $\h^3$.
 \item If $b=0$ and $a=1$, then $G$ is isometric to the product space $\h^2\times\R$.
 \item If $b=0$ and $a\neq 1$, then any totally umbilical surface in $G$ falls into one of the following cases:
\begin{itemize}
 \item  An integral surface of one of the distributions $\{E_1,E_3\}$ or $\{E_2,E_3\}$ (in particular, totally geodesic).
 \item A surface invariant by a $1$-parameter group of isometries associated to one of the Killing vector fields $\partial_x$ or $\partial_y$. This gives rise to two complete non totally geodesic surfaces, unique up to ambient isometries.
\end{itemize}
 \item In the rest of cases, $G$ does not admit totally umbilical surfaces.
\end{enumerate}
\end{theorem}

\begin{proof}
If $a=0,$ we know (see Remark \ref{rmk:nounim-examples}) that $G$ is isometric to $\h^3,$ so we will assume $a\neq 0$. Let us first analyze the case $b\neq 0$. Equations~\ref{eqn:nounim-equation1} and~\ref{eqn:nounim-equation3} ensure that the angle functions $\nu_1,\nu_2,\nu_3$ of $\Sigma$ satisfy $P(\nu_1,\nu_2)=Q(\nu_1,\nu_2)=0$, where $P$ and $Q$ are the polynomials
\begin{equation}\label{eqn:nounim-polynomials}
\begin{aligned}
 P(x,y)&=((a+1)(a+2)y^2-(a-1)(a-2)x^2-2a)b+2(a^2-1)xy,\\
 Q(x,y)&=4bx^2y^2+ab((1-a)x^2-(1+a)y^2+2a)(1-x^2-y^2).
\end{aligned}
\end{equation}
Note that we have substituted $\nu_3^2=1-\nu_1^2-\nu_2^2$ in~\ref{eqn:nounim-equation3} to obtain the expression for $Q(x,y)$. We will suppose that if $P$ and $Q$ have a non-zero common factor $R$, and prove that $R$ is constant, so Bézout's theorem~\cite{Fulton} implies that there are only finitely-many solutions $(x,y)$ of the system $P(x,y)=Q(x,y)=0$. In particular, the angle functions must be constant.

In order to prove it, note that equation $P(0,y)=0$ with unknown $y$ has exactly two solutions, which are given by
\begin{align*}
y_1&=\frac{\sqrt{2a}}{\sqrt{a^2+3a+2}},&y_2&=\frac{-\sqrt{2a}}{\sqrt{a^2+3a+2}}.
\end{align*}
It is easy to compute
\[Q(0,y_1)=Q(0,y_2)=\frac{2a^2b(a^2+a+2)}{(a+2)^2}\neq 0,\]
so $R(0,y)\neq 0$ for all $y\in\R$. We also know that the degree of $R$ is at most $2$. This degree cannot be $2$ since, were it the case, $R$ would be a scalar multiple of $P$ but $P$ has zeroes along the axis $x=0$ and $R$ does not. If the degree of $R$ is one, then $R(x,y)=0$ represents a line in the $(x,y)$-plane so it must be parallel to the axis $x=0$; this means that  $R(x,y)=x-\mu$ for some $\mu\neq 0$. Equivalently, $P(\mu,y)=Q(\mu,y)=0$ for all $y\in\R$ and some $\mu\neq 0$, which is impossible in view of~\eqref{eqn:nounim-polynomials}. 

As pointed out before, this proves that $\nu_1,\nu_2,\nu_3$ are constant. Let us consider the following vector fields, which are tangent to $\Sigma$:
\begin{align*}
X&=\nu_2E_3-\nu_3E_2,&Y&=\nu_3E_1-\nu_1E_3.
\end{align*}
By using the umbilicity of $\Sigma$ and~\eqref{eqn:nabla-nounim}, we get
\begin{equation}\label{eqn:constant-angle}
\begin{aligned}
  0&=\langle\nabla_{X}N,E_1\rangle=(a\nu_3^2-\nu_2^2)b,\\
   0&=\langle\nabla_{Y}N,E_2\rangle=(a\nu_3^2+\nu_1^2)b.
\end{aligned}
\end{equation}
These two equations and the fact that $b\neq 0$ yield $\nu_1=\nu_2=0$ and $\nu_3^2=1$,  so $a=0$,  contradicting our original assumption.

Finally, let us deal with the case $b=0$, so the technique above is no longer valid since $Q$ identically vanishes. We assume $a\neq 1,$ otherwise $G$ is isometric to $\h^2\times \R$ (see Remark \ref{rmk:nounim-examples}). Under these assumptions, we get that $\nu_1\nu_2=0$ from~\eqref{eqn:nounim-equation3}.  Let us work in a open subset of $\Sigma$ where $\nu_1=0$ (the case $\nu_2=0$ is similar and is discussed below). 

In view of Section~\ref{sec:nouni-intro}, the metric of the ambient space is given by 
\[\df s^2= e^{-2(1+a)z}\df x^2 + e^{-2(1-a)z}\df y^2 + \df z^2.\] 
Clearly, the vector fields $\partial_x$ and $\partial_y$ are Killing fields, and the condition $\nu_1=0$ implies that $\partial_x$ is tangent to $\Sigma$, so the surface is invariant by the one parameter group of isometries $(x,y,z) \to (x+c, y, z)$. Discarding the trivial case of a totally geodesic plane $\{y=y_0\}$ (which is an integral surface of the distribution spanned by $\{E_1,E_3\}$), we can assume that a piece of $\gamma$, which we still denote by $\gamma$, is a graph  over the $y$-axis, i.e., it is generated by a curve $\gamma$ in the totally geodesic plane $\{x=0\}$ given by $\gamma(y)=(0,y,z(y))$, and the generated surface is parametrized by
\begin{equation*}
  X(x,y)= (x,y,z(y)).
\end{equation*}
Hence, $X_x= e^{-(1+a)z }E_1$ and $X_y= e^{-(1-a)z} E_2 + z^\prime E_3$, so we can take as a unit normal vector field
 \begin{equation*}
 N= \frac{e^{-az} z^\prime}{D} E_2 - \frac{e^{-z}}{D} E_3,
 \end{equation*}
 where $D^2={e^{-2az} (z')^2 + e^{-2z}}$.  From~\eqref{eqn:nabla-nounim}, it is straightforward to compute
\begin{equation*}
\nabla_{X_x}N= \frac{(1+a) e^{-z}}{D} X_x,
\end{equation*}
so $X$ is a totally umbilical immersion if and only if
\begin{equation}\label{eqn:term1}
\nabla_{X_y}N =  \frac{(1+a) e^z}{D} X_y=\frac{(1+a) e^z}{D} \left(  e^{(1-a)z} E_2 + z^\prime E_3 \right).
\end{equation}
Independently, using \eqref{eqn:nabla-nounim}, we can calculate
\begin{equation}\label{eqn:term2}
\begin{aligned}
\nabla_{X_y}N&= \left((1-a)\frac{e^{-(2-a)z}}{D}+\left(\frac{e^{-az} z^\prime}{D}\right)^\prime\right) E_2\\
&\qquad\qquad+ \left(
(1-a) \frac{z'e^{-z}}{D}-\left(\frac{e^{-z}}{D}\right)'\right) E_3,
\end{aligned}
\end{equation}
Comparing~\eqref{eqn:term1} and~\eqref{eqn:term2}, we deduce that the umbilicity condition is equivalent to the following ODE system:
\begin{align}
\left(\frac{e^{-az} z^\prime}{D}\right)'&=2a\frac{e^{-(2-a)z}}{D},&
\left(\frac{e^{-z}}{D}\right)'&=-2a\frac{z'e^{-z}}{D}.\label{ODE}
\end{align}
Observe that, if $z=z(y)$ is a solution, then the function $z(\pm y+y_0)$ is also a solution for all $y_0\in\R$, so we can restrict to the solutions of~\eqref{ODE} defined on the maximal interval containing $0$ and such that $z'(0)\geq 0$ (note that this change of variable corresponds to an isometry in the ambient space). We claim that the following statements are equivalent:
\begin{itemize}
 \item[(a)] $z$ is a maximal solution of the system~\eqref{ODE}.
 \item[(b)] There exist $\Lambda,\theta\in\R$ with $\Lambda>0$ and $\theta\geq\frac{-1}{2a}\log\Lambda$, such that $z$ is, up to a change of parameter of the form $y\mapsto\pm y+y_0$, the unique solution of the ODE
\begin{equation}\label{eq:final}
z''= (3a-1)\Lambda^2 e^{2(3a-1)z}-(a-1)e^{2(a-1)z},
\end{equation}
with initial conditions $z(0)=\theta$ and $z'(0)=\sqrt{\Lambda^2 e^{2(3a-1)\theta}-e^{2(a-1)\theta}}$.
\end{itemize}
First, if (a) is satisfied, then integrating the second equation in~\eqref{ODE}, we obtain that $D= \Lambda e^{(2a-1)z}$ for some constant $\Lambda>0$. Combining this with the identity $D^2=e^{-2az} (z')^2 + e^{-2z}$ we get
\begin{equation}\label{eqn-derivative}
(z')^2 = \Lambda^2 e^{2(3a-1)z}-e^{2(a-1)z}.
\end{equation}
The fact that the RHS must be positive implies that $\theta=z(0)\geq\frac{-1}{2a}\log\Lambda$, and also $z'(0)^2=\Lambda^2 e^{2(1+a)\theta}-e^{2(1-a)\theta}$. Finally, using that $D= \Lambda e^{(2a-1)z}$ in the first equation of~\eqref{ODE}, it can be easily transformed into \eqref{eq:final}, so assertion (b) is proved. 

Reciprocally, let us suppose that $z$ is a solution of~\eqref{eq:final} for some $\Lambda$ and $\theta$, with initial conditions as in (b). Then, multiplying both sides in~\eqref{eq:final} by $2z'$, integrating, and imposing $z'(0)^2=\Lambda^2 e^{2(3a-1)\theta}-e^{2(a-1)\theta}$, we obtain that $z$ also satisfies~\eqref{eqn-derivative}. From here, we deduce that $D=\Lambda e^{(2a-1)z}$, which, together with~\eqref{eq:final}, implies the two equations in~\eqref{ODE}.

The claim is proved and, in particular, we deduce the following properties of the solutions of~\eqref{ODE}:
\begin{itemize}
 \item They are smooth and defined on the whole real line.
 \item They admit the global lower bound $z\geq\frac{-1}{2a}\log\Lambda$.
 \item They are not bounded from above. Were it not the case, there would be a sequence $\{y_n\}\to\infty$ such that $\lim\{z'(y_n)\}=0$ and converging to the upper bound $M<\infty$, so~\eqref{eqn-derivative} would say that $M=\lim\{z(y_n)\}=\frac{-1}{2a}\log(\Lambda)$, and $z$ would be constant, which is not a solution of~\eqref{ODE}.
 \item There exists $y_0\in\R$ such that $z(y_0)=\frac{-1}{2a}\log\Lambda$. Otherwise, we get that $z$ is strictly monotonic since $z'$ has no zeroes in view of~\eqref{eqn-derivative}. This implies that there is a sequence $\{y_n\}\to\pm\infty$ such that $\lim\{z''(y_n)\}=0$ and $\lim\{z'(y_n)\}=0$. Evaluating~\eqref{eqn-derivative} at $y_n$ and taking limits, we get that $\lim\{z(y_n)\}=\frac{-1}{2a}\log\Lambda$. Likewise, evaluating \eqref{eq:final} at $y_n$ and taking limits implies that $a\not\in[\frac{1}{3},1]$ and $\lim\{z(y_n)\}=\frac{-1}{2a}\log\left(\sqrt{\frac{3a-1}{a-1}}\Lambda\right)$. Both limits can only coincide for $a=0$, so we get a contradiction.
\end{itemize}
This implies that, up to an isometry in the ambient space, for each $\Lambda>0$, the totally umbilical surface can be supposed to be the surface associated to the entire solution $z_\Lambda$ of~\eqref{eq:final} with initial conditions $z(0)=\frac{-1}{2a}\log\Lambda$ and $z'(0)=0$. Moreover, by the uniqueness in terms of initial conditions, we deduce that the function $z_\Lambda$ is even (i.e., $z_\Lambda(-y)=z_\Lambda(y)$ for all $y\in\R$). 

The final step will consist in showing that the constructed surfaces are congruent under an ambient isometry when varying $\Lambda>0$. Given $\Lambda_1,\Lambda_2>0$, let $z_{\Lambda_1}$ and $z_{\Lambda_2}$ be the associated solutions, and consider $w=\frac{1}{2a}(\log\Lambda_1-\log\Lambda_2)$. It is easy to check that the transformation
\[(x,y,z)\mapsto(xe^{(1+a)w},ye^{(1-a)w},z+w)\]
defines an ambient isometry which maps the surface parametrized by $(x,y)\mapsto(x,y,z_{\Lambda_1}(y))$ into that parametrized by $(x,y)\mapsto(x,y,z_{\Lambda_2}(y))$.

Assume now $\nu_2=0$ in an open subset of $\Sigma,$  that is, this open subset is invariant by the isometries $(x,y,z)\to (x,y+c,z).$ By the same argument (exchanging the roles of $x$ and $y$ and replacing $a$ by $-a$), then $\Sigma$ is either part of a totally geodesic plane $\{x=x_0\}$, or is, up to an ambient isometry, generated by a curve $\gamma(x)=(x,0,z(x))$ where $z$ is a solution of some second-order ODE with prescribed initial conditions. 

It remains to check that this surface is not congruent to the first one, invariant in the $x$-direction. Indeed, note that if they were congruent, since the level curves of the umbilicity function are Euclidean lines parallel to the $x$-axis or to the $y$-axis depending on the case, we could left-translate two congruent level curves so that they pass through the identity element. The resulting surfaces would be congruent by an element in the stabilizer of the identity, which consists of the maps $(x,y,z)\mapsto(\pm x,\pm y,z)$, see~\cite[Proposition 2.21]{MPR}. It is clear that these transformations cannot interchange lines parallel to the $x$-axis and lines parallel to the $y$-axis, so we get a contradiction.
\end{proof}

\begin{remark}\label{rmk:family-existence}
 The proof of the case $b=0$ in Theorem~\ref{thm:nounimodular} is inspired by that of the case of $\Sol$ in~\cite{ST09}. It turns out that both cases could have been treated together as the semidirect products $\mathbb{R}^2\ltimes_A\mathbb{R}$, where
\[A=\left(\begin{matrix}1&0\\0&c\end{matrix}\right),\]
for some $c\in\R$. The (unimodular) case $c=-1$ is nothing but $\Sol$ endowed with its canonical metric, and the (non-unimodular) cases $c\neq -1$ are isometric, up to rescaling the metric, to the family of non-unimodular metric Lie groups $\mathbb{R}^2\ltimes_{A(a,b)}\mathbb{R}$ with $b=0$ (see~\cite[Proposition 2.24]{MPR}).

The reason why the two umbilical surfaces constructed in Theorem~\ref{thm:nounimodular} reduce to only one in $\Sol$, up to ambient isometries, is the fact that the ambient space for $c=-1$ also carries the symmetry $(x,y,z)\mapsto(-y,x,-z)$, which swaps the $x$-direction and the $y$-direction.
\end{remark}


\begin{thebibliography}{99}
  \bibliographystyle{alpha}

\bibitem{Chen}
B.~Y.~Chen.
\newblock Classification of totally umbilical submanifolds in symmetric spaces.
\newblock \emph{J. Austral. Math. Soc. Ser. A} \textbf{30} (1980/81), no. 2, 129--136.

\bibitem{Chen2}
B.~Y.~Chen.
\newblock Totally umbilical submanifolds.
\newblock \emph{Soochow J. Math.} \textbf{5} (1979), 9--37. 

\bibitem{Daniel}
B.~Daniel.
\newblock Isometric immersions into 3-dimensional homogeneous manifolds.
\newblock {\em Comment. Math. Helv.} \textbf{82} (2007), no. 1, 87--131.

\bibitem{Fulton}
W.~Fulton.
\newblock Algebraic curves, an introduction to geometry.
\newblock \emph{W.A. Benjamin, Inc.}, Mathematics Lecture Note Series, 1969.

\bibitem{Man}
J.~M.~Manzano.
\newblock On the classification of Killing submersions and their isometries.
\newblock Preprint, arXiv:1211.2115.

\bibitem{MPR}
W.~H.~Meeks, J.~Pérez.
\newblock Constant mean curvature surfaces in metric Lie groups.
\newblock In \emph{Geometric Analysis: Partial Differential Equations and Surfaces}, Contemporary Mathematics (AMS) vol. 570 (2012), 25--110.

\bibitem{MT}
B.~Mendonça, R.~Tojeiro.
\newblock Umbilical submanifolds of $\mathbb{S}^n\times\mathbb{R}$.
\newblock Preprint, arXiv:1107.1679.

\bibitem{Milnor}
J.~Milnor.
\newblock Curvatures of Left Invariant Metrics on Lie Groups.
\newblock \emph{Advances in Math.} \textbf{21} (1976), no. 3, 293--329.

\bibitem{ST09}
R.~Souam, E.~Toubiana.
\newblock Totally umbilic surfaces in homogeneous $3$-manifolds.
\newblock \emph{Comment. Math. Helv.} \textbf{84} (2009), no. 3, 673--704.

\bibitem{SV}
R.~Souam, J.~Van der Veken.
\newblock Totally umbilical hypersurfaces of manifolds admitting a unit Killing field.
\newblock \emph{Trans. Amer. Math. Soc.} \textbf{364} (2012), no. 7, 3609--3626.

\bibitem{Spivak}
M.~Spivak.
\newblock{A comprehensive introduction to Differential Geometry}.
\newblock \emph{Publish or Perish}. Boston, 1970.

\bibitem{Tsu96}
K.~Tsukada
\newblock Totally geodesic submanifolds of Riemannian manifolds and curvature-invariant subspaces.
\newblock \emph{Kodai Math J.} \textbf{19} (1996), no. 3, 395--437.

\bibitem{Veken}
J.~Van der Veken.
\newblock Higher order parallel surfaces in Bianchi-Cartan-Vranceanu spaces.
\newblock \emph{Results Math.} \textbf{51} (2008), no. 3--4, 339--359.



\end{thebibliography}
\end{document}